\documentclass[a4paper]{amsart}
\usepackage{graphicx}
\usepackage{amsmath,amsthm,amssymb}
\usepackage{enumitem}
\usepackage{setspace}
\usepackage{tikz}
\usetikzlibrary{cd}
\usepackage{xcolor}
\usepackage{mathrsfs}
\usepackage[colorlinks=true,citecolor=blue,linkcolor=blue]{hyperref}
\RequirePackage{cite}

\newcommand{\C}{\mathbb{C}}

\newcommand{\cohshv}[1]{\widetilde{#1}}

\newcommand{\faceof}{\preceq}

\newcommand{\mProj}{\Proj_{\textrm{MH}}}

\newcommand{\N}{\mathbb{N}}

\newcommand{\NReal}{N_{\R}}

\newcommand{\dual}[1]{{#1}^{\vee}}

\newcommand{\R}{\mathbb{R}}
\newcommand{\restrict}[2]{\left.#1\right|_{#2}}
\newcommand{\set}[2]{\left\{#1\,\middle|\,#2\right\}}

\newcommand{\strshf}[1]{\mathscr{O}_{#1}}
\newcommand{\supp}[1]{\left|#1\right|}
\newcommand{\TCDiv}{\CDiv_{T}}

\newcommand{\Z}{\mathbb{Z}}

\newcommand{\wrong}[1]{}

\DeclareMathOperator{\CDiv}{CDiv}
\DeclareMathOperator{\Hom}{Hom}
\DeclareMathOperator{\Pic}{Pic}

\DeclareMathOperator{\Proj}{Proj}

\DeclareMathOperator{\relint}{rel\,int}
\DeclareMathOperator{\SF}{SF}
\DeclareMathOperator{\Spec}{Spec}
\DeclareMathOperator{\tProj}{tProj}

\DeclareMathOperator{\im}{im}
\DeclareMathOperator{\Div}{div}
\numberwithin{equation}{section}

\newtheorem{theorem}[equation]{Theorem}
\newtheorem{lemma}[equation]{Lemma}
\newtheorem{corollary}[equation]{Corollary}
\newtheorem{proposition}[equation]{Proposition}

\theoremstyle{definition}
\newtheorem{definition}[equation]{Definition}
\newtheorem{notation}[equation]{Notation}
\theoremstyle{remark}
\newtheorem{remark}[equation]{Remark}
\newtheorem{example}[equation]{Example}

\newtheorem{assumption}[equation]{Assumption}
\title{Comparing two Proj-like constructions on toric varieties}
\author{Vivek Mohan Mallick}
\email{vmallick@iiserpune.ac.in}
\address{Department of Mathematics, IISER Pune, Dr Homi Bhabha Road, Pashan, Pune 411008, India}

\author{Kartik Roy}
\email{kartik.roy@students.iiserpune.ac.in}
\address{Department of Mathematics, IISER Pune, Dr Homi Bhabha Road, Pashan, Pune 411008, India}
\thanks{The second author is supported by CSIR.}
\date{\today}

\begin{document}

\begin{abstract}
    The paper explores the relation between Perling's toric Proj of a multigraded ring $A$ associated with a toric variety ($\tProj A$), and Brenner and Schr\"oer's homogeneous spectrum $\mProj A$ of the same ring. We show that there is always a canonical open embedding and study a criterion for them to be isomorphic.
\end{abstract}

\maketitle
%\begin{abstract}
%    This article establishes a relation between toric Proj known as $\tProj$ and homogeneous spectrum $\mProj$ of a multigraded ring.
%    Perling constructed $\tProj$ of a multigraded ring associted to certain toric varieties. On the other hand, Brenner and Schr\"oer constructed %homogeneous spectrum $\mProj$ for arbitrary multigraded rings. We show that there is a canonical open embedding from $\tProj A$  to $\mProj A$, %where $A$ is a canonical multigraded ring corresponding to certain toric varieties. We also give an isomorphism criterion for the embedding.
%\end{abstract}

\tableofcontents

\section{Introduction}

The homogeneous coordinate ring of a toric variety $X$, described by Cox \cite{cox:homcoordrng}, carries enough algebraic information about $X$ enabling one to reconstruct the toric variety.
This was demonstrated by Cox which was then reinterpreted in terms of a Proj-like construction by Perling \cite{perling:toricvarhomprimespec}, which we denote by $\tProj$.
These authors and Kajiwara \cite{kajiwara:funtorvar} also studied the relation between graded modules over the multigraded homogeneous coordinate ring and quasicoherent sheaves over the toric variety.
Cox's reconstruction of the toric variety involves taking quotients of a relevant open subset of the spectrum of the coordinate ring.
A similar quotient construction was introduced by Brenner and Schr\"oer \cite{breshr:ampfamilies} describing a multigraded Proj of a ring graded by a finitely generated abelian group, which we denote by $\mProj$.
This paper explores the relations between these constructions.

Projective spaces $\mathbb{P}^{n} $ and weighted projective spaces $\mathbb{P}(q_0,\dots,q_n)$  are homogeneous spectrum of $\N$-graded rings obtained using the classical $\Proj$ construction of projective varieties. 
A generalization of this construction to the homogeneous coordinate ring of a toric variety was implicit in Cox \cite{cox:homcoordrng} and Kajiwara \cite{kajiwara:funtorvar}, and was made explicit by Perling \cite{perling:toricvarhomprimespec}.
Though the concept of homogeneous coordinate ring $A$ of $X$ is not unique (see \cite{ahs:quotprtorvar}), we shall work with Cox's construction which is universal in the sense that every other homogeneous coordinate ring maps to it. If $A$ is a homogeneous coordinate ring  of $X$ then coherent sheaves of modules on $X$ correspond to graded $A$-modules (see \cite{cox:homcoordrng, perling:toricvarhomprimespec, ahs:quotprtorvar}). The cohomology of sheaves of modules on toric varieties $X$ using graded $A$-modules structure were studied by \cite{eisenbud:cohomtorvar} and \cite{mustata:vanthmtorvar}. 

%\section{Introduction}
Let $X_{\Delta}$ be a toric variety corresponding to fan $\Delta$ in the real vector space $\NReal$ for some lattice $N$. $X_{\Delta}$ is said to have enough effective invariant Cartier divisors if for each cone $\sigma \in \Delta$ the set $X_{\Delta} \setminus U_{\sigma}$ is the support of an effective invariant Cartier divisor. Now corresponding to $X_{\Delta}$ there exists a homogeneous coordinate ring $A$ graded by $\Pic(\Delta)$, the Picard group of $X_{\Delta}$. Perling \cite{perling:toricvarhomprimespec} and Kajiwara \cite{kajiwara:funtorvar} showed that $X_{\Delta}$ is isomorphic to $\tProj A$, which corresponds to homogeneous prime ideals in $A$ with a certain irrelevant locus removed. Meanwhile, $\mProj A$ is in bijection with homogeneous ideals (not necessarily prime) such that homogeneous elements in the complement of such ideals are multiplicatively closed. We present the following relation between these spaces.
\begin{theorem}[see theorem \ref{the:embdd}]
There is an equivariant open embedding $\mu : \tProj A \rightarrow \mProj A$.
\end{theorem}
Now assume $\Delta$ simplicial. A simplicial cone in $\Delta$ is a cone generated by a linearly independent subset of one-dimensional cones of $\Delta$. We say $\Delta$ is simplicially complete if it contains all simplicial cones generated by its rays. We prove the following.
\begin{theorem}
 $\mu : \tProj A \rightarrow \mProj A$ is an isomorphism if and only if $\Delta$ is simplicially complete.
\end{theorem}

Using this theorem we completely classify whether $\mu$ is an isomorphism or not when $X_{\Delta}$ is a complete nonsingular surface.

The paper has four sections including an introduction.

In the second section, we review the $\tProj$ construction of a multigraded ring from Perling \cite{perling:toricvarhomprimespec}. We recall some results about $\strshf{\tProj}$-modules from Kajiwara \cite{kajiwara:funtorvar}. We do not present any new results here.

The third section consists of a detailed description of multihomogeneous spaces, first described by Brenner and  Schr\"oer \cite{breshr:ampfamilies}. We discuss some properties of multihomogeneous spaces and recall the functor $\widetilde{(\ )}$ from  the category of graded $A$-modules to the category of $\strshf{\mProj A}$-modules. We construct the functor $\Gamma_{*}(\mProj A, \underline{\ })$ from the category of $\strshf{\mProj A}$-modules to the category of graded $A$-modules. We use this to show that $\widetilde{(\ )}$ is essentially surjective.

In the fourth section, we establish the open embedding $\mu$ from $\tProj$ to $\mProj$ and give a criterion for the embedding to be an isomorphism.

\section{Some preliminary notions and results}
\subsection{Perling's construction ($\tProj$)}

This section is a review of Perling's reconstruction
\cite{perling:toricvarhomprimespec} of a toric variety $X$ from a
multigraded ring associated with $X$ via a generalization of the Proj
construction, which he denotes by $\tProj$.

A toric variety $X$ is a normal variety with an action of a
dense open torus such that the torus action on the variety restricts to the
multiplication on the torus.
The category of toric varieties is equivalent to
the category of pairs $(N, \Delta)$ where $N$ is a lattice (free abelian
group of finite rank) and $\Delta$ is a fan in the real vector space
$N_{\R}$.
The toric variety corresponding to $\Delta$ will be denoted by $X_{\Delta}$.
We denote the $\Z$-dual of $N$ by $M$ .
See \cite{fulton:toric} and \cite{cls:toricvar} for the results on toric
varieties used here.

We have the following assumption on the fans associated with the toric
varieties. Recall that support of a fan is the union of all its cones.

\begin{assumption}\label{asm:fullfan}
 The support of $\Delta$ generates the vector space $\NReal$. 
\end{assumption}

Define $\Delta(1) := \{\rho \in \Delta: \dim(\rho) = 1\}$ the set of
$1$-dimensional cones in $\Delta$ and we denote $n_\rho$ the primitive
element of $\rho$.
For each $\rho \in \Delta(1)$, we denote the associated
$T$-invariant prime Weil divisor on $X_{\Delta}$ by $D_{\rho}$.
$\TCDiv(X_{\Delta})$ denotes the free abelian group generated by
$T$-invariant Cartier divisors on $X_{\Delta}$.
We further assume the following
\begin{assumption} \label{asm:frpicgp}
  The Picard group $\Pic(\Delta)$ of $X_{\Delta}$ is free.
\end{assumption}

The Picard group fits into an exact sequence (see \cite[section
3.4]{fulton:toric} or \cite{oda:torembed}),
\begin{equation*} \label{eq:tcdiv}
    1 \xrightarrow{} M \xrightarrow{\Div} \TCDiv(\Delta) \xrightarrow{\deg}
    \Pic(\Delta) \xrightarrow{} 1. 
\end{equation*}

\begin{definition}
  A $\Delta$-linear support function $h$ is a function from the support
  $\supp{\Delta}$ to $\R$ such that $h$ is linear on each cone $\sigma \in
  \Delta$ and sends integral points in $\supp{\Delta}$ to $\Z$.
\end{definition}

We denote by $\SF(\Delta)$ the free abelian group of finite rank of $\Delta$
linear support functions.
It is well known (see \cite[section 3.4]{fulton:toric}) that such functions
correspond to the Cartier divisors.
% $\{m_{\sigma}\} \subset M $ such that ${m_{\sigma} - m_{\sigma'}} \in
% {(\sigma \cap \sigma')}^{\perp}$.
Therefore $\TCDiv(\Delta) \cong \SF(\Delta)$ is an isomorphism of abelian groups.
Denote the real vector space of support functions by $\SF(\Delta)_{\R}:=
\SF(\Delta) \otimes_{\Z} \R$. There is a split short exact sequence of
abelian groups
\begin{equation} \label{eq:splexseq}
    1 \xrightarrow{} M \xrightarrow{\Div} \SF( \Delta) \xrightarrow{\deg}
    \Pic(\Delta) \rightarrow{} 1,
\end{equation}
and a corresponding  exact sequence of tori  
 \[
 1 \xrightarrow{} G  \xrightarrow{} \Tilde{T}  \xrightarrow{} T  \xrightarrow{} 1,
 \]
where
$G := \Spec \C[\Pic(\Delta)],
\Tilde{T} := \Spec \C[\SF( \Delta)]$
and $T := \Spec k[M]$.

\begin{definition}\label{def:suppsuppfun}
  For a support function $h \in \SF(\Delta)_{\R}$, we define its support as $\supp{h} :=
  \{\rho \in \Delta(1): h(n_{\rho}) \not= 0\}$.
\end{definition}

For each ray $\rho \in \Delta(1)$, the spaces $H_{\rho} := \{h \in \SF(
\Delta)_{\R}: h(n_\rho) \geq 0\} \subset \SF(\Delta)_{\R}$ are half spaces
\cite[proposition 3.2]{perling:toricvarhomprimespec} whose boundaries
$\partial H_\rho$ are rational hyperplanes.

Our results, which depend on Perling and Kajiwara's results, require the
condition of having enough invariant Cartier divisors, which we define now.

\begin{definition}\cite[proposition 3.3]{perling:toricvarhomprimespec}
  A toric variety $X_{\Delta}$, corresponding to a fan $\Delta$ in a lattice
  $N$, \emph{has enough invariant Cartier divisors} if for each $\sigma \in
  \Delta$ there exists an effective $T$-invariant Cartier divisor whose
  support is precisely the union of $D_\rho$ for $\rho \notin \sigma(1)$.
\end{definition}

\begin{remark}
  For cone $\sigma$ if $D$ is such a Cartier divisor and it corresponds to the support function $h$ then
  $\supp{h}= \Delta(1) \setminus \sigma(1)$. This definition agrees with that of
  Kajiwara \cite[definition 1.5]{kajiwara:funtorvar} in terms of good
  cones.
\end{remark}

\begin{example}
  Simplicial toric varieties have enough invariant Cartier divisors (see
   \cite[lemma 3.4]{cox:homcoordrng}). 
  % Projective toric varieties provide another class of examples.
\end{example}

\begin{assumption}
  \label{ass:enufcartier}
  We assume the toric variety $X_{\Delta}$ has enough invariant Cartier
  divisors.
\end{assumption}
The set $C = \dual{\left(\bigcap_{\rho \in \Delta(1)} H_\rho\right)}$ is
a pointed strongly convex rational polyhedral cone in $\dual{\SF( \Delta)_{\R}}$.
For each $\rho \in \Delta(1)$, let $l_\rho$ be the primitive element of the
ray orthogonal to $H_\rho$ in $\dual{\SF( \Delta)_{\R}}$ which takes
positive values on $C$.
Define $\hat{\sigma}$ to be the cone generated by the $l_{\rho}$
corresponding to the rays in $\sigma$, i.e.\ $\langle \set{l_\rho}{\rho \in
\sigma(1)} \rangle$.
Then $\hat{\Delta} = \set{\hat{\sigma}}{\sigma \in \Delta}$ is a subfan of
$C$ (see proof of \cite[proposition 3.6]{perling:toricvarhomprimespec}) and 
the map of fans $(\dual{\SF( \Delta)_{\R}}, \hat{\Delta})
\longrightarrow (\NReal, \Delta)$ is surjective.
This induces the following diagram of toric morphisms
\begin{equation*}
  \begin{tikzcd}
    X_{\hat{\Delta}} \arrow[r, "\phi"] \arrow[d, "\pi"'] & X_C \\
    X_{\Delta} &
  \end{tikzcd}
\end{equation*}
where $\phi$ is an open morphism and $\pi$ is a quotient presentation in the
sense of \cite{ahs:quotprtorvar}.

\begin{remark}[cf.\ remark 3.8 of \cite{perling:toricvarhomprimespec}]
  \label{rem:simcon}
  When the fan $\Delta$ is simplicial then the quotient presentation $\pi :
  X_{\hat{\Delta}} \rightarrow X_{\Delta}$ is same as described by Cox
  \cite{cox:homcoordrng}.  Further $\hat{\Delta}$ is simplicial if and only
  if $\Delta$ is simplicial.
 \end{remark}

\begin{notation}
  \label{not:tprojgitfan}
The coordinate ring $A := \C[\dual{C} \cap \SF(\Delta)] $ of the affine
toric variety $X_C$ is $\Pic(\Delta)$-graded $\C$-algebra given by the
homomorphism $\deg$ in \ref{eq:splexseq}. For a support function $h \in
\SF(\Delta)$ we denote $\chi(h)$ the corresponding homogeneous element in
$A$. Let $\omega$ be the corresponding weight cone in the real vector space
$\Pic(\Delta)_{\R}$.
\end{notation}

Recall that $X_{\hat{\Delta}} = \bigcup_{\hat{\sigma} \in \hat{\Delta}}
X_{\hat{\sigma}}$.
Let $B_\sigma$ be the defining ideal for $X_C \setminus X_{\hat{\sigma}}$
and $B := \sum_{\sigma \in \Delta}B_{\sigma} \subset A$.
Then, $B$ has codimension at least 2 in $A$ and $V(B) = X_C \setminus
X_{\hat{\Delta}}$ proving that $\Gamma(X_{\hat{\Delta}},
\strshf{X_{\hat{\Delta}}}) = A$.
Furthermore, the dual action of $G$ on $A$ induces the isotypical
decomposition
\[
  A = \bigoplus_{\alpha \in \Pic(\Delta)} A_{\alpha}.
\]

Under the order reversing correspondence between faces of $C$ and faces of
$\dual{C}$, $\hat{\sigma} \in \hat{\Delta}$ corresponds to
$\Sigma_{\sigma}:= \dual{C} \cap {\hat{\sigma}}^{\perp}$.
For each $\hat{\sigma} \in \hat{\Delta}$, we fix a support function
$h_{\sigma} \in \relint(\Sigma_{\sigma})$.
Then
\begin{equation*}
  \supp{h_{\sigma}} = \Delta(1) \setminus \sigma(1). 
\end{equation*}
For each $\sigma \in \Delta$, let $\chi(h_{\sigma}) \in A$ be the element
corresponding to $h_{\sigma}$.
Let $A_{(\chi(h_{\sigma}))}$ be the subring of degree zero elements in
$\Pic(\Delta)$-graded localized ring $A_{\chi(h_{\sigma})}$.
Then we have $A^{G}_{\chi(h_{\sigma})} \cong A_{(\chi(h_{\sigma}))} \cong
k[\sigma_M]$ where $\sigma_M := \dual{\sigma} \cap M$ for $\sigma \in \Delta$
\cite[lemma 3.10]{perling:toricvarhomprimespec}.
Thus, $(X_{\Delta}, \pi)$ is a categorical quotient of the action of $G$ on
$X_{\hat{\Delta}}$ \cite[proposition 3.11]{perling:toricvarhomprimespec}.

As a topological space, 
\begin{equation*}
  \tProj A = \set{\wp \in X_{\hat{\Delta}}}{\wp \text{ is a homogeneous
  prime ideal in $A$}}.
\end{equation*}
Let $i \colon \tProj A \longrightarrow X_{\hat{\Delta}}$ be the canonical embedding.
One constructs a sheaf of rings $\strshf{}'$ on $X_C$ which for an open set
$U$, is the ring $\strshf{}'(U)$ consisting of those sections of
$\strshf{}(U)$, which locally are fractions of homogeneous elements of
degree $0$.
Perling \cite[definition 3.15]{perling:toricvarhomprimespec} defines $\strshf{\tProj A} = i^{-1}\strshf{}^{'}$ on $\tProj A$.
The ringed space $\tProj A = (\tProj A, \strshf{\tProj A})$ is called the
toric proj of $A$, and Perling proves that this is indeed a scheme.
He shows that if $\tau < \sigma$ in $\Delta$ and $h_\tau$ and $ h_\sigma$
are support functions vanishing on $\tau(1)$ and $\sigma(1)$, respectively
then we have the following diagram (see \cite[theorem
3.18]{perling:toricvarhomprimespec})
 \begin{equation} \label{fig:comdig}
    \begin{tikzcd}
             \tProj A_{(\chi(h_{\tau}))} \arrow[hook]{d} \arrow{r}{\cong} & D_{+}(\chi(h_{\tau})) \arrow[hook]{d} \arrow{r}{\cong} & \Spec A_{(\chi(h_{\tau}))} \arrow{r}{\cong} \arrow[hook]{d} & \Spec k[\tau_M]   \arrow[hook]{d}   \\
         \tProj A_{(\chi(h_{\sigma}))}\arrow{r}{\cong}  &
	 D_{+}(\chi(h_{\sigma})) \arrow{r}{\cong} & \Spec
	 A_{(\chi(h_{\sigma}))} \arrow{r}{\cong}  & \Spec k[\sigma_M].
    \end{tikzcd}
\end{equation}
This establishes that $\tProj A$ is a scheme and is isomorphic to the toric
variety. In fact, it presents $X_{\Delta}$ as a geometric quotient of
$X_{\hat{\Delta}}$ by $G$ \cite[theorem 3.19]{perling:toricvarhomprimespec}.
   
The advantage of this presentation is that there is an essentially
surjective functor from the category of $\Pic(\Delta)$-graded $A$-modules to
quasi-coherent $\strshf{\tProj}$-modules. More explicitly, for a graded
$A$-module $F$, the $\widetilde{(\ )}$ construction
\cite[definition 3.21]{perling:toricvarhomprimespec} gives a
quasi-coherent $\strshf{\tProj}$-module $\widetilde{F}$ on $\tProj A$. For
each $\alpha \in \Pic{(\Delta)}$, we have the graded $A$-modules $A(\alpha)$
with decomposition $A(\alpha) := \oplus_{\beta \in
\Pic(\Delta)}A(\alpha)_{\beta}$ where $\beta \in \Pic(\Delta)$ and
$A(\alpha)_{\beta}:= A_{\alpha + \beta}$. The associated quasi-coherent
$\strshf{\tProj}$-module $\strshf{\tProj A}(\alpha) : = \Tilde{A(\alpha)}$
is an invertible sheaf \cite[proposition 2.6(1)]{kajiwara:funtorvar} for
each $\alpha \in \Pic(\Delta)$. Like projective spaces, for $\alpha , \beta
\in \Pic(\Delta)$ we have the isomorphisms $\strshf{\tProj A}(\alpha +
\beta) \cong \strshf{\tProj A}(\alpha) \otimes_{\strshf{\tProj
A}}\strshf{\tProj A}(\beta)$ \cite[corollary 2.8]{kajiwara:funtorvar} and
$\widetilde{F(\alpha)} \cong \widetilde{F} \otimes_{\strshf{\tProj A}}
\strshf{\tProj A}(\alpha) $ for a graded $A$-module $F$
\cite[corollary 2.9]{kajiwara:funtorvar}. 
    
On the other hand, for a quasi-coherent $\strshf{\tProj}$-module
$\mathscr{F}$, there is a graded $A$-module $\Gamma_{*}(\tProj
A,\mathscr{F}) := \bigoplus_{\alpha \in \Pic(\Delta)} \Gamma(\tProj A,
\mathscr{F}(\alpha))$ where $\mathscr{F}(\alpha) := \mathscr{F}
\otimes_{\strshf{\tProj A}} \strshf{\tProj A}(\alpha)$ with
$\widetilde{\Gamma_{*}(\tProj A,\mathscr{F})} \cong \mathscr{F}$. The
functors $\widetilde{(\ )}$ and $\Gamma_{*}(\tProj A, \underline{\ })$  are adjoint functors
while the former one is left adjoint, and the later one is right adjoint.
  
\subsection{Multihomogeneous Spaces}
This subsection is a  review the theory of multihomogeneous spaces. For more
details regarding multihomogeneous spaces, we refer to \cite[section
2]{breshr:ampfamilies}. See also \cite{MR3307753} for results on the geometry of
multigraded algebras and their properties.

\begin{definition}
  \label{def:prdrlelt}\label{def:muhospcs}
  Let $A = \bigoplus_{d \in D} A_d$ be a ring graded by a finitely generated
  abelian group $D$. A homogeneous element $f \in A$ is relevant if the
  $D$-graded ring $A_f$ (with the induced grading) is periodic, i.e.\ the
  subgroup of $D$ consisting of all degrees of homogeneous invertible
  elements is of finite index. Let $A_{(f)}$ denote the degree zero part of
  $A_f$. Then, (see \cite[lemma 2.1]{breshr:ampfamilies}) $D_+(f) := \Spec A_f
  \longrightarrow \Spec A_{(f)}$ turns out to be a geometric quotients. Brenner
  and Schr{\"o}er (loc. cit.) defines
  \begin{equation*}
    \mProj A = \bigcup_{
      \substack{
	f \in A \\
	f \text{ is relevant}
      }
    } D_+(f)
  \end{equation*}
  where the union takes place in the universal quotient in the category of
  ringed spaces.
\end{definition}

For further properties, we refer to \cite{vk:mhstvar}. We collect some results
which we will find useful.
\begin{remark}
  \label{rem:ptsinmhs}
  The points in a multihomogeneous projective space $\mProj A$ of a
  $D$-graded ring $A$ correspond to homogeneous ideals in A which may not be
  prime (see \cite[remark 2.3]{breshr:ampfamilies}). However, these ideals
  have the property that all the homogeneous elements in the complement form
  a multiplicatively closed set.
\end{remark}
 
\begin{remark}
  \label{rem:prjempty}
  The way $\mProj A$ is defined for a $D$-graded ring $A$, it can happen
  that $A$ has no relevant element and then $\mProj A = \emptyset$. If
  $A$ is a finitely generated algebra over $A_0$, one sufficient condition
  for the existence of relevant elements is that there exists a collection
  of homogeneous generators $\set{x_i}{1 \leq i \leq r}$ such that
  $\set{\deg x_i}{1 \leq i \leq r}$ generates a finite index subgroup in
  $D$. This condition is easy to check, for example, when $A$ is the
  polynomial ring over $\C$.
\end{remark}

\begin{remark}
  \label{rem:frelgmqt}
  By \cite[lemma 2.1]{breshr:ampfamilies},  the map $\Spec A_f
  \longrightarrow \Spec A_{(f)}$, which is induced by the inclusion $A_{(f)}
  \hookrightarrow A_f$, is a geometric quotient.
\end{remark}

By definition, the collection of affine open subschemes
\[
\set{D_+(f)}{f \in A \text{ is homogeneous and relevant}}
\]
covers $\mProj A$. We state the following easy fact for subsequent use.

\begin{lemma}
  \label{lem:dpfdpgfg}
  With the notation as above, $D_+(f) \cap D_+(g) = D_+(fg) \subset \mProj
  A$.
\end{lemma}

\begin{proof}
  This is implicit in \cite[proposition 3.1]{breshr:ampfamilies}. Note that
  for relevant elements $f$ and $g$ in $A$, $\Spec A_{fg} = \Spec A_f \cap
  \Spec A_g$ as subschemes of $\Spec A$. Now $\Spec A_{(fg)}$, $\Spec
  A_{(f)}$ and $\Spec A_{(g)}$ are geometric quotients (see remark
  \ref{rem:frelgmqt}) under the action of $\Spec A_0[D]$ and hence $\Spec
  A_{(fg)} = \Spec A_{(f)} \cap \Spec A_{(g)}$ considered as a subscheme of
  $\mProj A$.
\end{proof}

For later, we record two results of Brenner and Schr{\"o}er regarding
finiteness.

\begin{lemma} \cite[lemma 2.4]{breshr:ampfamilies}
  \label{lem:noethcnd}
  Let $D$ be a finitely generated abelian group. A $D$-graded ring $A$
  is noetherian if and only if $A_0$ is noetherian and $A$ is an $A_0$ algebra
  of finite type.
\end{lemma}

\begin{proposition} \cite[proposition 2.5]{breshr:ampfamilies}
  \label{pro:unclfntp}
  Suppose $A$ is a noetherian ring graded by a finitely generated abelian
  group $D$. Then the morphism $\varphi \colon \mProj A \longrightarrow
  \Spec A_0$ is universally closed and of finite type.
\end{proposition}

\begin{definition}  \cite[section 3]{breshr:ampfamilies},
  \label{def:simtorembed}
  Let $R$ be a ring, $M$ be a free abelian group of finite rank, and $N:=
  \Hom(M,\Z)$ be dual of $M$. Let $X$ be an $R$-scheme and $T:=\Spec R[M]$
  be the torus. A \emph{simplicial torus embedding} of torus $T$ is
  $T-$equivariant open map $T \hookrightarrow X$ locally given by semigroup
  algebra homomorphisms $R[\dual{\sigma} \cap M] \rightarrow R[M]$, where
  $\sigma$ is a strongly convex simplicial cone in $N$.
\end{definition}
 
 \begin{example}
   If $X$ is a simplicial toric variety, then $X$ is a simplicial torus embedding of the underlying torus. Apart from these, multihomogeneous spaces of $\Z^{n}$-graded polynomial algebras are another example.
 \end{example}
 
 Let $D$ be an abelian group of finite type and $A = k[x_1, \dots, x_n]$ be
 a $D-$graded polynomial $k-$algebra. Suppose the grading is given by a
 linear map $P: \Z^{n} \rightarrow D$ with finite co-kernel. Then we have
 the following exact sequence of abelian groups
 \begin{equation*}
     0 \rightarrow M \rightarrow Z^{n} \rightarrow D,
 \end{equation*}
  where $M$ is the kernel of $P$.
  
  \begin{proposition} \cite[proposition 3.4]{breshr:ampfamilies}
  Assume the above setting. Then $\mProj A$ is a simplicial torus embedding
  of the torus $\Spec k[M]$.
  \end{proposition}
  
  \begin{remark} \cite[remark 3.7]{breshr:ampfamilies} \label{rem:relconecorr}
    Again we assume the above setting. Let $I = \{1, \dots, n\}$ be an index
    set and $N:= \Hom_{\Z}(M,\Z)$ be the dual of $M$. Let $\text{pr}_i:
    \Z^{n} \rightarrow \Z, i \in I$ be projections. To each subset $J
    \subset I$ we associate the cone $\sigma_J \subset \NReal$ generated by
    $\text{pr}_i|_{M}, i \in J$. Then $J \subset I$ with $\prod_{i \in
    J}x_i$ relevant in $A$ correspond to strongly convex simplicial cone
    $\sigma_{I\setminus J} \subset \NReal$. Let $\overline{\Delta}$ be the
    set of all strongly convex simplicial cones $\sigma_{J}$ such that
    $\prod_{i \in I\setminus J}x_i$ is relevant. In general,
    $\overline{\Delta}$ is not a fan in $\NReal$. 
  \end{remark}

\section{Sheaves associated to multigraded modules}
Consider a finitely generated abelian group $D$ and let $A$ be a $D$-graded
ring. Suppose $M = \bigoplus_{d \in D} M_d$ is a $D$-graded $A$-module. Just
as in the case of quasicoherent sheaves of modules over $\Proj$ of a $\N$-graded ring
\cite[definition before proposition 5.11, page 116]{hartshorne:alggeom},
we can construct $\cohshv{M}$.

 Since the points $p$ in
$\mProj A$ correspond to graded ideals $I_p$ in $A$ such that the homogeneous
elements in the complement $A \setminus I_p$ form a multiplicatively closed
set, it is still true that the stalk of the structure sheaf at $p$ is given
by $A_{(I_p)}$ (see remark \ref{rem:ptsinmhs}). One can now define
$\widetilde{M}$ in the same way by associating to $U \subset \mProj A$, the
$\strshf{\mProj A}(U)$-module of sections $s \colon U \to \coprod_{p \in U}
M_{(I_p)}$ satisfying the usual condition that locally such $s$ should be
defined by a single element of the form $m / a$ with $m \in M$ and
$a \in A$ but not in any of the ideals $I_p$. These modules are coherent
under some mild conditions, as we state below. Note that, given a $D$-graded
$A$-module $M= \bigoplus_{d \in D} M_d$ and an $e \in D$, one can define a
graded module $M(e)$ whereas $A$-modules $M(e) = M$, but $M(e)_d =
M_{d + e}\ \forall d \in D$.

\begin{lemma}
  \label{lem:cohshvms}
  Suppose $D$ is a finitely generated abelian group and $A$ is a $D$-graded
  integral noetherian ring. Then for $X = \mProj A$, the following hold
  \begin{enumerate}[nosep, label=(\textit{\alph*})]
    \item \label{ite:AtldisOX}
      $\widetilde{A} = \strshf{X}$. This allows us to define
      \begin{equation*}
	\strshf{X}(d) := \widetilde{A(d)}.
      \end{equation*}
      $\strshf{X}(d)$ is a coherent sheaf.
    \item \label{ite:MDpfMftd}
      For a $D$-graded $A$-module $M$, $\widetilde{M}$ is
      quasi-coherent and $\restrict{\widetilde{M}}{D_+(f)} \cong
      \widetilde{M_{(f)}}$ for any relevant element $f \in A$, where $\widetilde{M_{(f)}}$
      is the sheaf of modules over $\Spec A_{(f)}$ corresponding to the module $M_{(f)}$,
      the degree zero elements in $M_f$. Moreover, $\widetilde{M}$ is coherent whenever $M$ is finitely generated.
      \item The functor  $M \rightarrow \widetilde{M}$ is a covariant exact functor from the category of $D$-graded $A$-modules to the category of quasi-coherent $\strshf{X}$-modules, and commutes with direct limits and direct sums. 
  \end{enumerate}
\end{lemma}

The proof follows almost by definition and is very similar to the proof of
\cite[proposition 5.11]{hartshorne:alggeom}.  
\begin{remark}
Lemma \ref{lem:cohshvms} is known for toric varieties with enough invariant
Cartier divisors \cite[\S 2]{kajiwara:funtorvar}. This includes simplicial
toric varieties.
\end{remark}

\begin{remark}
Note that we have used $M$ as a lattice as well as $A$-module. Meaning should be clear from the context.
\end{remark}
\begin{remark}
  In general, the functor $\widetilde{(\ )}$ is not faithful, even for
  projective varieties, but it is full as we shall see later in proposition
  \ref{pro:tildefull}.
\end{remark}

The following three results are from \cite[section 3]{vk:mhstvar}.
\begin{theorem}
  \label{the:glsecdAd}
  Suppose $D$ is a \emph{free} finitely generated abelian group and $A =
  \bigoplus_{d \in D} A_d$ is a $D$-graded \emph{integral domain} which is
  finitely generated by homogeneous elements $x_1, \dotsc, x_r \in A$ over
  the ring $A_0$. Also assume that for all $k$, $1 \leq k \leq r$, the set
  $\set{\deg x_i}{1 \leq i \leq r, i \neq k}$ generates a finite index
  subgroup of $D$. Let $X = \mProj A$. Then $\Gamma(X, \strshf{X}(d)) \cong
  A_d$. Furthermore, $\strshf{X}(d)$ is a reflexive sheaf.
\end{theorem}

\begin{theorem}
Assume the above hypothesis. If $M$ is a $D$-graded $A$-module then
$\Gamma(X, \widetilde{M(d)}) \cong M_d$ for each $d \in D$.
\end{theorem}

\begin{theorem}
  \label{the:criOXdlb}
  Suppose $X = \mProj A$ is a multihomogeneous space defined for a
  $D$-graded integral domain $A = \bigoplus_{d \in D} A_d$ generated by
  homogeneous elements $x_1, \dotsc, x_r$ over $A_0$. 
  Moreover assume that $A_0$ is a field and $A^{\times} = A_0^{\times}$. Let
  $d \in D_f$, a sublattice of $D$
  generated by $ \set{\deg a}{a \text{ divides } f^N \text{ for some } N > 0}
  $, for every relevant element $f \in A$. Then
  $\strshf{X}(d)$ is a line bundle.
\end{theorem}
\begin{remark} \label{rem:lbdsubsp}
Assume the above hypothesis. Let $S$ be a collection of relevant elements such that $D_f =D$ for each $f \in S$. Then for each $d \in D$, the restriction of $\strshf{X}$-module $\strshf{X}(d)$ is a line bundle on $\bigcup_{f \in S} D_{+}(f)$.
\end{remark}
\begin{remark}
Let $X_{\Delta}$ be a toric variety with enough invariant effective Cartier divisors. We shall see later that for each cone $\sigma \in \Delta$, the corresponding homogeneous element $\chi(h_{\sigma})$ is relevant and satisfies $D_{\chi(h_{\sigma})} =\Pic(\Delta)$. So the $\strshf{\tProj A}$-modules $\strshf{\tProj A}(\alpha)$, $\alpha \in \Pic(\Delta)$, are line bundles on $\tProj A = \bigcup_{\sigma \in \Delta} D_{+}(\chi(h_{\sigma}))$ which recovers the result \cite[proposition 2.6(1)]{kajiwara:funtorvar}.
\end{remark}

\begin{definition}
Let $X = \mProj A$. For a quasi-coherent $\strshf{X}$-module $\mathcal{F}$, we define 
\begin{equation*}
  \Gamma_{*}(X, \mathcal{F}) := \bigoplus_{d \in D} \Gamma(X, \mathcal{F}(d))
\end{equation*}
where $\mathcal{F}(d) = \mathcal{F} \otimes_{\strshf{X}} \strshf{X}(d), d \in D$.
\end{definition}

\begin{remark}
Lemma \ref{the:glsecdAd} implies that $A \cong \Gamma_{*}(X, \strshf{X})$.
\end{remark}

$\Gamma_{*}(X,\underline{\ })$ is a covariant functor from quasi-coherent $\strshf{X}$-modules to $D$-graded $A$-modules.

\begin{proposition}\label{pro:tildefull}
Let $\mathcal{F}$ be a quasi-coherent $\strshf{X}$-module. Then the homomorphism $\mu: \widetilde{\Gamma_{*}(X, \mathcal{F})} \rightarrow \mathcal{F}$ is an isomorphism. In fact, every quasi-coherent $\strshf{X}$-module is of the form $\widetilde{F}$ for some $D$-graded $A$-module $F$. 
\end{proposition}
\begin{proof}
  The proof is similar to the proof of  \cite[theorem 3.2]{cox:homcoordrng}
  which itself is a generalization of \cite[proposition
  II.5.15]{hartshorne:alggeom}.  
\end{proof}
\begin{remark}\label{rem:adjunc}
Recall that $\widetilde{(\ )}$ and $\Gamma_{*}(\tProj A,\underline{\ })$ are adjunctions over $\tProj A$. However, they are not adjunctions over $\mProj A$.  
\end{remark}
 \section{Comparison of the $\tProj$ and $\mProj$}

 \subsection{An open embedding}
\begin{lemma} \label{lem:wgcone}
  The weight cone $\omega$ (described in \ref{not:tprojgitfan}) is a  full
  dimensional pointed strickly convex rational polyhedral cone in $\Pic(\Delta)_{\R}$.
\end{lemma}
\begin{proof}
 It is clear from the fact that (\ref{eq:splexseq}) is split exact sequence and $\dual{C}$ surjects onto $\omega$.
\end{proof}

\begin{definition} [cf definition 1.3(2) in \cite{kajiwara:funtorvar}]\label{def:locSF}
For each $\sigma \in \Delta$, $\SF(\check{\sigma})$ denotes the group of integral support functions on $\Delta$ with support precisely  $ \check{\sigma} := \Delta(1) \setminus \sigma(1)$.
\end{definition}

\begin{lemma}[cf lemma 1.7(2) in \cite{kajiwara:funtorvar}] \label{lem:relcone}
Let $\sigma \in \Delta$. Then $\deg(\SF(\check{\sigma})) = \Pic(\Delta)$.
\end{lemma}

\begin{proof}
The proof follows from computing the rank of the following short exact sequence
\begin{equation*}
   1 \xrightarrow{} M \cap \sigma^{\perp}   \xrightarrow{\Div} \SF(\check{\sigma}) \xrightarrow{\deg} \Pic(\Delta) \rightarrow{} 1,  
\end{equation*}
induced from \ref{eq:splexseq}.
\end{proof}

\begin{proposition}
Let $X_{\Delta}$ be the toric variety associated to the fan $\Delta$ in $\NReal$. Then $\chi(h_{\sigma})$ is a relevant element in $A$, the algebra of support functions (defined in \S 2), with respect to $\Pic(\Delta)$-grading for all $\sigma \in \Delta.$
\end{proposition}
\begin{proof}
For the $\Pic(\Delta)$-graded $\C$-algebra $A_{\chi(h_{\sigma})}$, the subgroup generated by degrees of units  $D' = \deg(\SF(\check{\sigma})) = \Pic(\Delta)$ by lemma \ref{lem:relcone}. Therefore $\chi(h_{\sigma})$ is a relevant element in $A$ by definition \ref{def:prdrlelt} for all cones $\sigma \in \Delta$.
\end{proof}
 
 \begin{theorem} \label{the:embdd}
   Let $X_{\Delta}$ be the toric variety with enough invariant Cartier divisors
   associated with the fan $\Delta$ and $\Spec(\C[M])$ its torus. Then there is a $\Spec(\C[M])$ equivariant open embedding
   $\mu: \tProj A \hookrightarrow \mProj A$, where $A$ is the algebra of
   support functions on $\Delta$ (defined in \ref{not:tprojgitfan}).
 \end{theorem}
 
 \begin{proof}
   Recall that $\mProj A := \bigcup_{f \colon f \text{ is relevant in } A}
   \Spec A_{(f)}$ and $\chi(h_{\sigma})$ is relevant for all $\sigma \in
   \Delta$.
   For each $\sigma \in \Delta$, we have
   $U_{\sigma} \cong \Spec k[\sigma_M] \cong \Spec (A_{(\chi(h_{\sigma}))})$
   by \cite[lemma 3.10]{perling:toricvarhomprimespec}
   and for $\tau \faceof \sigma$, we have the following commutative diagram:
 \begin{equation} \label{equ:opmapdiag}
     \begin{tikzcd}
       \tProj(A_{(\chi(h_{\tau}))}) \arrow[r, "\cong"] \arrow[d,hook] &
       \Spec(k[\tau_M]) \arrow[r,"\cong"]\arrow[d,hook] &
       \Spec(A_{(\chi(h_{\tau}))}) \arrow[d,hook]\\
       \tProj(A_{(\chi(h_{\sigma}))}) \arrow[r, "\cong"] &
       \Spec(k[\sigma_M]) \arrow[r,"\cong"] & \Spec(A_{(\chi(h_{\sigma}))}).
     \end{tikzcd}
 \end{equation}
 It follows from the above diagram that $X_{\Delta}$ is isomorphic to $\tProj A$. Moreover, the morphism $\mu$, which is the composition of the following morphisms, 
 \begin{equation}
     \tProj A \cong \bigcup_{\sigma \in \Delta}
     \tProj(A_{(\chi(h_{\sigma}))}) \cong \bigcup_{\sigma \in \Delta} \Spec(A_{(\chi(h_{\sigma}))}) \hookrightarrow \mProj A
 \end{equation}
 is an open morphism. Note that the morphisms in \ref{equ:opmapdiag} are $\Spec(\C[M])$ equivariant. This makes $\mu$ a $\Spec(\C[M])$ equivariant open morphism.
 \end{proof}
  
  \subsection{Criterion for isomorphism}
  In this subsection, we give a criterion for the open embedding to be an
  isomorphism.
  
  Let $X_{\Delta}$ be a simplicial toric variety associated to the fan $\Delta$ in
  $\NReal$ satisfying assumptions \ref{asm:fullfan} and \ref{asm:frpicgp},
  i.e.\ the fan does not lie in a lower dimensional subspace and the Picard
  group $\Pic(\Delta)$ of $X$ is free.
  % By remark \ref{rem:polalg}, $A$ is a polynomial $\C$-algebra with
  % indeterminates $ \chi(h_{\rho}), \rho \in \Delta(1)$.
  Then we have the short exact sequence \ref{eq:splexseq}. 
 
 Keeping in mind remark \ref{rem:relconecorr}, we define the following:
  \begin{definition}\label{def:simcomfan}
   Let $\Delta$ be a simplicial fan in $N_{\R}$ and $\Delta(1)$ be the set of rays.
   \begin{enumerate}
       \item A simplicial cone in $\Delta$ is a cone $\tau \subset N_{\R}$ generated by $S$, a linearly independent subset of $\Delta(1)$. 
       \item $\Delta$ is said to be simplicially complete if it contains every simplicial cone in $\Delta$.
   \end{enumerate}
  \end{definition}
  \begin{example}
   The fans of projective and weighted projective spaces are simplicially complete but the fans of Hirzebruch surfaces $H_r, r\geq 1$ are not.
  \end{example}
  
  Let $\Delta$ be a simplicial fan in $\NReal$. For each ray $\rho \in \Delta(1)$, we define an integral support function which takes
  the following values on the rays
  \begin{equation}
        K_{\rho}: \supp{\Delta} \rightarrow \R, \qquad
  K_{\rho}(n_{\rho'}) =  \begin{cases}
   \lambda_{\rho}, & \mbox{ if } \rho' = \rho\\
  0, & \mbox{ if } \rho' \not= \rho.
  \end{cases}
  \end{equation}

  We choose $\lambda_{\rho} \in \Z_{\geq 0}$ in a way such that $K_{\rho}$ are primitive element in the lattice $\SF(\Delta)$. Furthermore, $K_{\rho}$ are linearly independent and are contained in $\dual{C} \cap \SF(\Delta)$. For a fan $\Delta$ (not necessarily simplicial) in $\NReal$ we define the map $\Phi: \SF(\Delta)_{\R} \rightarrow \R^{\Delta(1)}$ by sending a support function $ h $ to $ (h(n_{\rho}))_{\rho \in \Delta(1)}$. Then $\Phi$ is an injective linear map and the following diagram commutes
  
  \begin{equation}
      \begin{tikzcd}
        \SF(\Delta) \arrow[r,"\Phi"] \arrow[d] & \Z^{\Delta(1)} \arrow[d] \\
        \SF(\Delta)_{\R} \arrow[r,"\Phi"] & \R^{\Delta(1)} 
      \end{tikzcd}
  \end{equation}
    $\Phi$ has finite cokernel whenever $\Delta$ is simplicial.
  \begin{lemma}\label{lem:polalg}
  Let $X_{\Delta}$ be a toric variety corresponding to the fan $\Delta$ in $\NReal$ satisfying \ref{asm:fullfan} and $A$ be the coordinate ring of the affine toric variety $X_{C}$ (defined in \ref{not:tprojgitfan}). Then $\Delta$ is simplicial if and only if $A$ is a polynomial $\C$-algebra of Krull dimension $\dim(\SF(\Delta)_{\R})$. 
  \end{lemma}
  \begin{proof}
  Assume $\Delta$ is simplicial.  Note that $\{K_{\rho}\}_{\rho \in \Delta(1)}$ forms a basis of $\SF(\Delta)$. If $h \in \SF(\Delta)$  and $\Phi(h) = (c_{\rho})_{\rho \in \Delta(1)} \in \Phi(\SF(\Delta)) \cap (\Z_{\geq 0})^{\Delta(1)}$, then  $c_{\rho} = h(n_{\rho}) = a_{\rho}\lambda_{\rho}$ implying $h = \sum_{\rho \in \Delta(1)} a_{\rho}K_{\rho}$. Therefore $\langle K_{\rho} : \rho \in \Delta(1) \rangle = \dual{C} \cap \SF(\Delta) = \Phi(\SF(\Delta)) \cap (\Z_{\geq 0})^{\Delta(1)}$ which implies $A = \C[\dual{C} \cap \SF(\Delta)]$ is a polynomial $\C$-algebra. It is clear that Krull dimension of $A$ is $\dim(\SF(\Delta)_{\R})$.  \\
    
    On the other hand, assume that $A$ is a polynomial $\C$-algebra of Krull dimension $\dim(\SF(\Delta)_{\R})$. It is then clear that there exists a basis for the semigroup $\dual{C} \cap \SF(\Delta)$ which makes $\dual{C}$ a simplicial cone. Therefore the cone $C$ and  its faces $\hat{\sigma}, \text{ where } \sigma \in \Delta$, are simplicial. Finally remark \ref{rem:simcon} implies the cones $ \sigma \in \Delta$ are simplicial.
  \end{proof}
  
  \begin{remark} \label{rem:algpresn}
  When $\Delta$ is simplicial the $\C$-algebra $A = \C[\chi(K_{\rho}) : \rho \in \Delta(1)]$. For each $\sigma \in \Delta$ we can take $h_{\sigma} = \sum_{\rho \in \sigma(1)}K_{\rho}$ and therefore $\chi(h_{\sigma}) = \prod_{\rho \in \sigma(1)}\chi(K_{\rho})$.
  \end{remark}
  
 \begin{theorem} \label{the:isom}
 Let $X_{\Delta}$ be a simplicial toric variety corresponding to the fan $\Delta$ in $\NReal$ satisfying \ref{asm:fullfan} and $A$ be the coordinate ring of the affine toric variety $X_{C}$ (defined in \ref{not:tprojgitfan}). Then $\Delta$ is simplicially complete if and only if  the morphism $\mu: \tProj A \rightarrow \mProj A$ in \ref{the:embdd} is an isomorphism.
 \end{theorem}
 \begin{proof}
 Note that $\mathcal{B} := \{K_{\rho}\}_{ \rho \in \Delta(1)}$ is a basis of $\SF(\Delta)$ and for each $\rho$, $\{K_{\rho'}\}_{\rho \not=\rho' \in \Delta(1)}$ is a basis of the boundary $\partial H_{\rho}$. With respect to $\mathcal{B}$, the projections in \ref{rem:relconecorr} are given by $pr_{\rho}  = r_{\rho}K^{*}_{\rho} = l_{\rho} \in \dual{\SF(\Delta)}, \rho \in \Delta(1)$ where $r_{\rho} \in \R_{> 0}$ and $K^{*}_{\rho} \in \dual{\SF(\Delta)_{\R}}$ with $K^{*}_{\rho}(K_{\rho}) = 1$. Then the restrictions $pr_{\rho}|_{M} = \im(l_{\rho}) = n_{\rho} \in N$. Let $S$ be a set of relevant monomials such that $\mProj A = \cup_{f \in S} D_{+}(f)$. Then by remark \ref{rem:relconecorr} each $f \in S$ corresponds to the simplicial cone 
 \[\sigma_f = \langle pr_{\rho} : \rho \in \Delta(1) \setminus \supp{f} \rangle = \langle n_{\rho} : \rho \in \Delta(1) \setminus \supp{f} \rangle  
 \] 
 over $\Delta$, and since  $\Delta$ is simplicially complete, we have $\sigma_f= \sigma$ for some $\sigma \in \Delta$. Therefore $f = \chi(h_{\sigma})$ for each $f \in S$ and hence $\tProj A = \cup_{\sigma \in \Delta} D_{+}(\chi(h_{\sigma})) = \mProj A$.\\
 
 Now assume $\mu$ is an isomorphism. Then we have $\tProj A \cong \cup_{\sigma \in \Delta} D_{+}(\chi(h_{\sigma})) \cong \mProj A$. Suppose $\Delta$ is not simplicially complete and $\sigma$ is a simplicial cone over $\Delta$, not contained in $\Delta$. Then $\sigma$ corresponds to the relevant monomial $f = \prod_{\rho \in \Delta(1) \setminus \sigma(1)}\chi(K_{\rho})$. The homogeneous prime ideal $\mathscr{P}= (\chi(K_{\rho}) : \rho \in  \sigma(1))$ is contained in $\Spec A_{(f)}$. For any $\tau \in \Delta$ the corresponding relevant monomial $g = \prod_{\rho \in \Delta(1) \setminus \tau(1)}\chi(K_{\rho})$ takes zero at the point $\mathscr{P}$ and therefore $\mathscr{P} \not \in \Spec A_{(g)}$. Hence we get $\mathscr{P} \in \Spec A_{(f)} \subset \mProj A$ and $\mathscr{P} \not \in \cup_{\sigma \in \Delta} D_{+}(\chi(h_{\sigma})) = \mProj A$, a contradiction. Therefore $\Delta$ is simplicially complete.

 \end{proof}

 The following are some easy implications of the preceding theorem
 \ref{the:isom}.
 
\begin{corollary} \label{cor:surfnonisom}
Let $X$ be a simplicial toric surface associated to $\Delta$ in $N_{\R}$ satisfying \ref{asm:fullfan}. Then the morphism $\mu : X \rightarrow \mProj A$ in \ref{the:embdd} is not an isomorphism if and only if there exists a simplicial cone $\tau$ over $\Delta$, not contained in $\Delta$, with $\sigma \subset \tau$ for some $\sigma \in \Delta$.
\end{corollary}
\begin{proof}
$\mu$ not being an isomorphism means $\Delta$ is not simplicially complete. In the forward direction, we can take $\sigma$ to be a one-dimensional cone. The opposite direction is obvious.
\end{proof}

\begin{corollary} \label{cor:torsurfhcl}
Let $X$ be a simplicial toric surface associated to $\Delta$ in $N_{\R}$ satisfying \ref{asm:fullfan}.  If $\text{rank}(\Pic(\Delta)) \geq 3$ then the morphism $\mu: X \rightarrow \mProj A$ in theorem \ref{the:embdd} is not an isomorphism.
\end{corollary}
\begin{proof}
The hypothesis implies cardinality of $\Delta(1)$ is at least $5$. Note that each cone in $\Delta$ is generated by at most $2$ one-dimensional cones. There are $3$ one-dimensional cones in less than $180^{\circ}$. This says $\Delta$ is not simplicially complete, in other words, $\mu$ is not an isomorphism.
\end{proof}

\begin{remark}
Let $X$ be a simplicial toric surface associated to $\Delta$ in $N_{\R}$ satisfying \ref{asm:fullfan}. One can show that the condition  $\text{rank}(\Pic(\Delta)) \geq 3$ is equivalent to the condition that the fan $\Delta$ has at least $5$ one-dimensional cones.
\end{remark}

 Recall that, all complete nonsingular toric surfaces are gotten by successive blow-ups of  either $\mathbb{P}^2$ or Hirzebruch surfaces $H_r, r \geq 0.$

\begin{corollary}
Let $X$ be a complete nonsingular toric surface corresponding to fan $\Delta$, $A$ be the $\C$-algebra of support functions on $\Delta$ and $\mProj A$ the associated multihomogeneous space.
\begin{enumerate} [label = (\roman*)]
    \item $X = \mathbb{P}^2 :$ Theorem \ref{the:isom} implies $\mu : X \rightarrow \mProj A$ is an isomorphism since the fan is simplicially complete.
    \item $X = H_0 = \mathbb{P}^{1} \times \mathbb{P}^{1} :$ Same argument as in $(i)$.
    \item $X = H_r, r \geq 1 :$  It is obvious to see that the fan of $X$ is not simplicially complete, so by theorem \ref{the:isom} $\mu: X \rightarrow \mProj A$ is not an isomorphism.
    \item  $X$ is a (successive) blowup of $ \mathbb{P}^2 \text{ or } H_{r}, r \geq 0 :$ Same argument as in $(iii)$.
\end{enumerate}
\end{corollary}

\bibliographystyle{alpha}
\bibliography{references}
\end{document}